\documentclass[12pt,a4paper]{article}

\usepackage{amsmath,amsthm,amssymb,url,comment}

\newtheorem{theorem}{Theorem}
\newtheorem{lemma}[theorem]{Lemma}

\providecommand{\abs}[1]{\lvert#1\rvert}
\newenvironment{proofof}[1]{\noindent{\itshape
    #1. }}{\hfill$\qed$\medskip}
\newcommand{\dist}{\mbox{\rm dist\/}}
\newcommand{\ce}{{\cal E}}


\begin{document}

%%%%%%%%%%%%%%%%%%%%%%%%%%%%%%%%%%%%%%%%%%
\begin{comment}
\begin{frontmatter}

\title{Metric Spaces in Which Many Triangles Are Degenerate}

\author[label1]{Va\v sek Chv\' atal\fnref{fn1}}
%\ead{chvatal@cse.concordia.ca}
\author[label2]{Ida Kantor\fnref{fn2}}

\affiliation[label1]{Department of Computer Science and Software Engineering, Concordia University, Montreal (Emeritus) and 
 Department of Applied Mathematics, Charles University, Prague (Visiting).}

\fntext[fn1]{{\em E-mail address:} {\urlstyle{tt} chvatal@cse.concordia.ca}\\
Supported by NSERC  grant RGPIN/5599-2014 and by H2020-MSCA-RISE project CoSP-GA No. 823748.}

\affiliation[label2]{Computer Science Institute of Charles University, Prague.}

\fntext[fn2]{ {\em E-mail address:} {\urlstyle{tt} ida@iuuk.mff.cuni.cz}\\ 
Supported by GA\v{C}R grant 22-19073S and by Charles
University project UNCE/SCI/004.}

\end{comment}
%%%%%%%%%%%%%%%%%%%%%%%%%%%%%%%%%%%%%%%%%%%%%%5
%%

%\begin{document}
\normalsize

\title{{\bf \Huge Metric spaces in which many triangles are degenerate}}
%%% \title{{\bf \huge Difference representations of graphs}}
%...... check that no other paper shares the name
%...... think of a name

%\pagestyle{myheadings} \markright{{\small{\sc V. Chv\'{a}tal, I.~Kantor:  Metric Spaces in Which Many Triangles Are Degenerate }}}  %%%% ???

\author{Vašek Chvátal\thanks{Department of Computer Science and Software Engineering, Concordia University, Montreal (Emeritus) and 
 Department of Applied Mathematics, Charles University, Prague (Visiting). {\em E-mail:} {\urlstyle{tt} chvatal@cse.concordia.ca}. Supported by NSERC  grant RGPIN/5599-2014 and by H2020-MSCA-RISE project CoSP-GA No. 823748.} \hspace{5mm} Ida Kantor%
		\thanks{Computer Science Institute of Charles University, Prague. {\em E-mail:} {\urlstyle{tt} ida@iuuk.mff.cuni.cz}. Supported by GA\v{C}R grant 22-19073S and by Charles
University project UNCE/SCI/004.} }

\date{}
%\date{
% Charles University, Prague\\
%(e-mail: \texttt{ida@iuuk.mff.cuni.cz})
%}

%\footnotetext{% \hskip -.6 cm
%  \emph{Keywords and Phrases}: metric spaces, lines, De Bruijn---Erd\H{o}s theorem.\\

 %\emph{2010 Mathematics Subject Classification}:
  %05C62, 05C80.
	%\hfill    {\tt  [{\jobname}.tex]}\newline
 %Submitted to  \emph{SIDMA}   \hfill
 %Printed on \today 
%}

%\maketitle

%\renewcommand{\thefootnote}{\empty}

%\begin{document}
%*** Fix captions of figures.
\maketitle

%\author{Va\v sek Chv\' atal and Ida Kantor}

%\maketitle

\begin{abstract}
Richmond and Richmond ({\sc American Mathematical Monthly} {\bf 104} (1997), 713--719) proved the following theorem: If, in a metric space with at least five points, all triangles are degenerate, then the space is isometric to a subset of the real line. We prove that the hypothesis is unnecessarily  strong: In a metric space on $n$ points, $\binom{n}{3}-n+5$ arbitrarily placed or $3\binom{n-2}{2}+1$ suitably placed degenerate triangles suffice.
\end{abstract}
%\end{frontmatter}
%%%%%%%%%%%%%%%%%%%%%%%%%%%%%%%%%%%%%%%%%%%%%%%%%%%%%%%%%%

\section{Results.}
Given a metric space $(V,\dist)$, we follow \cite{ACH} in writing $[rst]$ to signify that 
$r,s,t$ are pairwise distinct points of $V$ and $\dist(r,s)+\dist(s,t)=\dist(r,t)$. Following~\cite{RiRi97}, we refer to three-point subsets of $V$ as {\em triangles;\/} if $[rst]$, then the triangle $\{r,s,t\}$ is called {\em degenerate.\/} 

Now let $(V,\dist)$ be a metric space.  
Trivially, if there is a linear order $\preceq$ on $V$ such that $r\prec s\prec t \;\Rightarrow [rst]$, then all triangles in $V$ are degenerate. Richmond and Richmond~\cite{RiRi97} proved the converse under a mild lower bound on $\abs{V}$:
\begin{theorem}[\cite{RiRi97}]\label{riri}
Let $(V,\dist)$ be a metric space such that $\abs{V}\ge 5$. If all triangles in $V$ are degenerate, then there is a linear order $\preceq$ on $V$ such that $r\prec s\prec t \;\Rightarrow [rst]$.
\end{theorem}
Here, the lower bound on $\abs{V}$ cannot be reduced: consider $V=\{a,b,c,d\}$ and
$\dist(a,b)\!=\!dist(b,c)\!=\!\dist(c,d)\!=\!\dist(d,a)\!=\!1$, $\dist(a,c)\!=\!\dist(b,d)\!=\!2$.\\

The purpose of this note is to prove that the hypothesis of Theorem~\ref{riri} can be relaxed as soon as $\abs{V}=6$ and that it can be relaxed further and further as $\abs{V}$ gets larger and larger. To state these results, let us call a set $\ce$ of three-point subsets of a set $V$ an {\em anchor in $V$\/} if, for every metric space $(V,\dist)$, the assumption that all triangles in $\ce$ are degenerate implies a linear order $\preceq$ on $V$ such that $r\prec s\prec t \;\Rightarrow [rst]$. In this terminology, Theorem~\ref{riri} asserts that whenever $\abs{V}=n\ge 5$, the set of all $\binom{n}{3}$ three-point subsets of $V$ is an anchor in $V$.

\begin{theorem}\label{weak}
If $\abs{V}=n$, then every set of $\binom{n}{3}-n+5$ three-point subsets of $V$ is an anchor in $V$.
\end{theorem}
The $\binom{n}{3}-n+5$ in Theorem~\ref{weak} cannot be replaced by $\binom{n}{3}-n+4$. To see this, consider the graph with vertices 
$1,2,\ldots n$ and edges $\{1,3\}$, $\{1,4\}$, $\{2,3\}$, $\{2,4\}$ and $\{i,i+1\}$ with $i=4,5,\ldots n-1$. In the metric space induced by this graph (in the usual way, where edges have unit lengths), the only nondegenerate triangles are $\{1,2,i\}$ with $i=5,6,\ldots n$.
\begin{theorem}\label{strong}
If $\abs{V}=n\ge 5$, then there is an anchor in $V$ consisting of $3\binom{n-2}{2}+1$ three-point subsets of $V$. 
\end{theorem}
We do not know whether or not the $3\binom{n-2}{2}+1$ in Theorem~\ref{strong} can be reduced.

\section{Proofs.}

Our arguments involve the following algorithm that, given a set $\ce$ of triangles in $V$, produces a certain sequence $T_1,T_2,\ldots T_m$ of pairwise distinct triangles outside $\ce$: 
\begin{quote}
Set $m=0$.\\ While some six-point subset $S$ of $V$ contains precisely $19$ members of $\ce \cup\{T_i:1\le i\le m\}$,\\ increment $m$ by one and then let $T_{m}$ be the $20$th three-point subset of $S$.
\end{quote}
In an iteration of this algorithm, more than one $S$ may be available, and so there may be more than one candidate for $T_m$. Nevertheless, candidates that are rejected now remain available in the next iteration and so, when the algorithm terminates, the set $\{T_i:1\le i\le m\}$ is uniquely determined. We let ${\rm cl}\,\ce$ denote its union with $\ce$. The role of this notion is explained by the following lemma.

\begin{lemma}\label{act}
Let $(V,\dist)$ be a metric space and let $\ce$ be a set of triangles in $V$. If all triangles that belong to $\ce$ are degenerate, then all triangles that belong to ${\rm cl}\,\ce$ are degenerate.
\end{lemma}

Most of the work involved in proving Lemma~\ref{act} is subsumed in its following special case:
\begin{lemma}\label{uf}
Let $(V,\dist)$ be a metric space such that $\abs{V}= 6$. If $19$ triangles in $V$ are degenerate, then all $20$ triangles in $V$ are degenerate.
\end{lemma}
\begin{proof}
By assumption, there is a triangle $T$ in $V$ such that the $19$ triangles distinct from $T$ are degenerate. In particular, for each of the three elements $\lambda$ of $T$, all triangles in $V-\{\lambda\}$ are degenerate and so, 
by Theorem~\ref{riri}, there is a linear order $\preceq_\lambda$ on $V-\{\lambda\}$ such that 
\begin{equation}\label{order}
r\prec_{\lambda}s\prec_{\lambda}t \;\Rightarrow [rst].
\end{equation}
Having chosen a $\mu$ in $T$, let us label the elements of $V-T$ as $u,v,w$ in such a way that 
$u\prec_\mu v\prec_\mu w$, and so $[uvw]$. Now for each of the remaining two elements $\nu$ of $T$, property~\eqref{order} implies $u\prec_\nu v\prec_\nu w$ or $w\prec_\nu v\prec_\nu u$; reversing the orders if necessary, we may assume that $u\prec_\nu v\prec_\nu w$, and so 
\begin{equation}\label{anchor}
u\prec_\lambda v\prec_\lambda w\;\text{ for all three $\lambda$ in $T$.}
\end{equation}
Next, let us label the elements of $T$ temporarily as $a,b,c$ in such a way that $a\prec_cb$ and then permanently as $x,y,z$: If $b\prec_ac$, then $x=a,y=b,z=c$; else either $x=a,y=c,z=b$ (in case $a\prec_bc$) or $x=c,y=a,z=b$ (in case $c\prec_ba$). Now we have
\begin{equation}\label{xyz}
x\prec_zy \;\text{ and }\; y\prec_xz.
\end{equation}
For future reference, let us note that
\begin{equation}\label{ident}
\text{the restrictions of $\preceq_x$ and $\preceq_z$ on $\{u,v,w,y\}$ are identical.}
\end{equation}
(Analogous statements apply to $x,y$ in place of $x,z$ and to $y,z$ in place of $x,z$, but we will not need these variations.) To see this, observe that, by virtue of \eqref{anchor} and \eqref{order}, our metric space determines the rank of $y$ in the restriction of $\preceq_x$ on $\{u,v,w,y\}$ in the same way as it determines the rank of $y$ in the restriction of $\preceq_z$ on $\{u,v,w,y\}$. 

The remainder of the proof relies on the fact that
\begin{equation}\label{men}
[\alpha\beta\gamma]\text{ and }[\alpha\gamma\delta]\;\;\Rightarrow\;\;[\alpha\beta\delta]\text{ and }[\beta\gamma\delta],
\end{equation}
which has been pointed out by Menger~\cite{M28}. (By the way, extensions of~\eqref{men} are discussed in~\cite[Section 6]{Chv04}.)

Finally, we are ready to prove that the triangle $\{x,y,z\}$ is also degenerate. Actually, we are going to prove $[xyz]$. For this purpose, let us distinguish between three cases.\\

\mbox{\hspace{1cm}}\textsc{Case 1:} $u\prec_zx$. By \eqref{xyz}, we have $u\prec_zx\prec_zy$, and so $[uxy]$; by \eqref{ident} and \eqref{xyz}, we have $u\prec_xy\prec_xz $, and so $[uyz]$. Now we have $[xyz]$ by~\eqref{men}.\\

\mbox{\hspace{1cm}}\textsc{Case 2:} $z\prec_xw$. By \eqref{xyz}, we have $y\prec_xz\prec_xw$, and so $[yzw]$; by \eqref{xyz} and \eqref{ident}, we have $x\prec_zy\prec_zw$, and so $[xyw]$. Now we have $[xyz]$ by~\eqref{men}.\\

\mbox{\hspace{1cm}}\textsc{Case 3:} $x\prec_zu\prec_zv\prec_zw$ and $u\prec_xv\prec_xw\prec_xz$. In particular, 
\[
[xuv],[xvw],[vwz],[uvz].
\]
Here we have neither $[vxz]$ (else $[uvz]$ and \eqref{men} would imply $[uvx]$, contradicting $[xuv]$)
nor $[xzv]$ (else $[xvw]$ and \eqref{men} would imply $[zvw]$, contradicting $[vwz]$); since 
$\{x,v,z\}$ is degenerate, we must have 
\[
[xvz].
\]
If $v\prec_xy$, then $[vyz]$ by \eqref{xyz}; else $y\prec_zv$ by \eqref{ident}, and so $[xyv]$ by \eqref{xyz}; in either case, this implies $[xyz]$ by $[xvz]$ and \eqref{men}. \end{proof}
\begin{proofof}{Proof of Lemma~\ref{act}}
By definition, members of ${\rm cl}\,\ce-\ce$ can be enumerated as $T_1,T_2,\ldots T_m$ in such a way that, for each $i=1,2,\ldots ,m$, some six-point subset $S_i$ of $V$ contains precisely $19$ members of $\ce \cup\{T_1, T_2, \ldots T_{i-1}\}$ and $T_i$. 
Induction on $i\;(=1,2,\ldots m)$ shows that all triangles belonging to $\ce \cup\{T_1, T_2, \ldots T_{i}\}$ are degenerate; the induction step relies on Lemma~\ref{uf}. \end{proofof}

The remaining proofs fit the following framework. 
Berge~\cite{Ber73} defined a {\em $3$-uniform hypergraph\/} as an ordered pair $(V,\ce)$ such that $V$ is a set and $\ce$ is a set of three-point subsets of $V$; elements of $V$ are called {\em vertices\/} and elements of $\ce$ are called {\em hyperedges.\/} Bollob\' as~\cite{Bol65} coined the term {\em $m$-saturated\/} to designate certain graphs and hypegraphs and later~\cite{Bol68} he introduced a related notion of {\em weakly $m$-saturated\/} graphs (which applies to hypergraphs, too). In the established terminology with notation far from unified, a $3$-uniform hypegraph $(V,\ce)$ is said to be {\em weakly $K^6_3$-saturated\/}~\cite[p.~97]{EFT91} or {\em weakly $K^3_6$-saturated\/}~\cite[p.~484]{Pik99} or {\em weakly $K^{(3)}_6$-saturated\/}~\cite{GHV18} if and only if ${\rm cl}\,\ce$ consists of all three-point subsets of $V$. We will adopt the notation of~\cite{Pik99}. The following lemma is reminiscent of the theorem asserting that all weakly $(d+2)$-saturated graphs are rigid~\cite[Theorem 1]{Kal84+}, where ``rigid'' has a geometric meaning that pertains to embedding these graphs into ${\bf R}^d$.

\begin{lemma}\label{kal}
All weakly $K^3_6$-saturated $3$-uniform hypergraphs are anchors.
\end{lemma}
\begin{proof}
Concatenation of Lemma~\ref{act} and Theorem~\ref{riri}.
\end{proof}

\begin{proofof}{Proof of Theorem~\ref{weak}}
Concatenation of the following lemma with Lemma~\ref{kal}.
\end{proofof}

\begin{lemma}\label{new}
Every $3$-uniform hypergraph with $n$ vertices and at least $\binom{n}{3}-n+5$ hyperedges is weakly $K^3_6$-saturated.
\end{lemma}
\begin{proof}
Consider a $3$-uniform hypergraph $(V,\ce)$ such that $\abs{\ce}\ge\binom{n}{3}-n+5$, where $n=\abs{V}$. Since 
$\abs{\ce}\le\binom{n}{3}$, we have $n\ge 5$; if $n=5$, then $\ce$ consists of all three-point subsets of $V$; now let us assume that $n\ge 6$. Given any three-point subset $T$ of $V$, we shall show that $T\in{\rm cl}\,\ce$. This is trivial when $T\in\ce$; to prove it when $T\not\in\ce$, 
it suffices to find a six-point subset $S$ of $V$ such that $T$ is the unique three-point subset of $S$ not belonging to
$\ce$.
For this purpose, take one vertex from each $T'-T$ such that $T'$ is a three-point subset of $V$ not belonging to $\ce$ and $T'\ne T$. With $W$ standing for the set of all these vertices, $T$ is the unique three-point subset of $V-W$ not belonging to $\ce$; since $\abs{W}\le \binom{n}{3}-\abs{\ce}-1$, we have $\abs{V-W}\ge n-\binom{n}{3}+\abs{\ce}+1\ge 6$.
\end{proof}
The remark that follows Theorem~\ref{weak} shows that the $\binom{n}{3}-n+5$ in Lemma~\ref{new} cannot be reduced.\\

\begin{proofof}{Proof of Theorem~\ref{strong}}
Concatenation of the following lemma with Lemma~\ref{kal}.
\end{proofof}

\begin{lemma}\label{cons}
For every integer $n$ greater than four there is a weakly $K^3_6$-saturated $3$-uniform hypergraph with $n$ vertices and 
$3\binom{n-2}{2}+1$ hyperedges.\\
\end{lemma}
\begin{proof}
Take a three-point subset $S$ of $V$ and let $\ce$ consist of all three-point subsets of $V$ that have a nonempty intersection with $S$. Note that $\binom{n}{3}-\binom{n-3}{3}=3\binom{n-2}{2}+1$. 
\end{proof}
It is a known fact that the $3\binom{n-2}{2}+1$ in Lemma~\ref{cons} cannot be reduced: see, for instance,~\cite{Fra82} or~\cite[Theorem 5.5]{Kal84}.

\section{Concluding remarks}

\begin{enumerate}
\item Theorems~\ref{riri},~\ref{weak}, and~\ref{strong} extend to the context of {\em pseudometric betweenness\/} defined in~\cite[p.~643]{BBC13}, which is more general than the context of metric spaces.

\item The proof of Lemma~\ref{new} extends to show that every $r$-uniform hypergraph with $n$ vertices and at least $\binom{n}{r}-n+k-1$ hyperedges is weakly $K^r_k$-saturated.

\item Following~\cite{BBC13}, let us refer to a $3$-uniform hypergraph $(V,\ce)$ as {\em metric\/} if there is a metric space $M$ such that a triangle in $M$ is degenerate if and only if it belongs to $\ce$. Lemma~\ref{uf} can be reformulated as the statement that the $3$-uniform hypergraph with $6$ vertices and $19$ hyperedges is non-metric. Actually, this hypergraph is minimal non-metric in the sense that the deletion of an arbitrary vertex from it produces a metric hypergraph. To verify this, observe first that the deletion produces a hypergraph with $5$ vertices and $10$  or $9$ hyperedges. The former hypergraph is obviously metric. To see that the latter hypergraph is metric, set $n=5$ in the comment that follows Theorem~\ref{weak}. (By the way, infinitely many minimal non-metric hypergraphs have been constructed in~\cite{CK22}.)
\end{enumerate}

%\begin{acknowledgment}{Acknowledgments.}
{\bf Acknowledgments.} 
We are grateful to Xiaomin Chen for pointing out that the type of hypergraphs that we are dealing with have been investigated before as weakly saturated graphs and for other insightful comments, and to Guillermo Gamboa for giving us the comment that follows Theorem~\ref{weak}.
%\end{acknowledgment}

\end{document}